\documentclass[3p,preprint]{elsarticle}
\usepackage[T1]{fontenc}
\usepackage{amsfonts}
\usepackage{amsmath}
\usepackage{geometry}
\usepackage{xcolor}

\newtheorem{theorem}{Theorem}[section]

\newtheorem{lemma}[theorem]{Lemma}

\newtheorem{remark}[theorem]{Remark}

\newenvironment{proof}[1][Proof]{\noindent\textbf{#1.} }{\ \rule{0.5em}{0.5em}}

\newcommand{\B}[1]{\boldsymbol{#1}}

\begin{document}

\bigskip

\bigskip
\begin{frontmatter}

\title{Two one-parameter families of nonconforming enrichments of the Crouzeix–Raviart finite element}
  \author[address-CS,address-Pau]{Federico Nudo\corref{corrauthor}}
  \cortext[corrauthor]{Corresponding author}
  \ead{federico.nudo@unical.it}

  \address[address-CS]{Department of Mathematics and Computer Science, University of Calabria, Rende (CS), Italy}
  \address[address-Pau]{Laboratoire de Mathématiques et de leurs Applications, UMR CNRS 5142, Université de Pau et des Pays de l'Adour (UPPA), 64000 Pau, France}

\begin{abstract}
In this paper, we introduce two one-parameter families of quadratic polynomial enrichments designed to enhance the accuracy of the classical Crouzeix--Raviart finite element. These enrichments are realized by using weighted line integrals as enriched linear functionals and quadratic polynomial functions as enrichment functions. To validate the effectiveness of our approach, we conduct numerical experiments that confirm the improvement achieved by the proposed method.
\end{abstract}


\end{frontmatter}

\section{Introduction}
The analysis of various physical phenomena often involves the widespread application of the finite element method, a powerful and versatile tool for approximating solutions of partial differential equations (PDEs)~\cite{ciarlet2002finite}. One of the reasons for its undeniable popularity lies in its ability to handle various types of geometries. In this method, the problem domain is divided into polytopes known as \textit{finite elements}. Each finite element represents a local solution of the differential problem, and their collective contributions construct a global piecewise solution. If the global piecewise solution remains continuous at the boundary of the subdomains, the finite element is referred to as \textit{conforming}; otherwise, it is referred to as \textit{nonconforming}.

A finite element is locally defined as the triple $\mathcal{M}_d=(S_d, \mathbb{F}_{S_d}, \Sigma_{S_d})$, where~\cite{Guessab:2022:SAB}
\begin{itemize}
       \item $S_d$ is a polytope in $\mathbb{R}^d$, $d\ge 1$,
       \item $\mathbb{F}_{S_d}$  is a vector space of dimension $n$ formed by real-valued functions defined on $S_d$, also known as \textit{trial functions},
       \item $\Sigma_{S_d}=\{L_j \, : \, j=1,\dots,n\}$ is a set of linearly independent linear functionals from the vector space $\mathbb{F}_{S_d}$, referred to as \textit{degrees of freedom},
\end{itemize}
ensuring that $\mathbb{F}_{S_d}$ is $\Sigma_{S_d}$-unisolvent, meaning that if $f\in \mathbb{F}_{S_d}$ and 
\begin{equation*}
      L_j(f)=0, \qquad j=1,\dots,n,
\end{equation*}
then $f=0$. To enhance the local approximation accuracy in the finite element method, a commonly employed technique involves introducing \textit{enrichment elements}~\cite{Guessab:2016:AADM, Guessab:2016:RM, guessab2017unified, DellAccio:2022:QFA, DellAccio:2022:AUE, DellAccio:2022:ESF, DellAccio2023AGC, DellAccio2023nuovo}. Specifically, to enrich the finite element $\mathcal{M}_d$, we introduce a set of enrichment functions ${e_1,\dots,e_N}$ and a set of enriched linear functionals
\begin{equation*}
\Sigma^{\mathrm{enr}}_{S_d}=\left\{L_j\, : \, j=1,\dots,n+N\right\},
\end{equation*}
ensuring that the approximation space $\mathbb{F}_{S_d}$ is $\Sigma^{\mathrm{enr}}_{S_d}$-unisolvent. We then define the enriched finite element
\begin{equation*}
\mathcal{M}^{\mathrm{enr}}_d=\left(S_d,\mathbb{F}^{\mathrm{enr}}_{S_d}, \Sigma^{\mathrm{enr}}_{S_d}\right),
\end{equation*}
where
\begin{equation*}
\mathbb{F}^{\mathrm{enr}}_{S_d}=\mathbb{F}_{S_d}\oplus\{e_1,\dots,e_N\}.
\end{equation*}
 The crucial aspect lies in the selection of the enrichment functions $e_1,\dots,e_N$ to ensure that the triple $(S_d, \mathbb{F}^{\mathrm{enr}}_{S_d}, \Sigma^{\mathrm {enr}}_{S_d})$ is a finite element.

The Crouzeix--Raviart finite element, named in honor of mathematicians Vivette Crouzeix and Pierre-Arnaud Raviart, is widely used in numerical analysis for numerically solving second-order elliptic PDEs. Its nonconforming nature proves highly effective in handling challenges posed by intricate geometries and irregular meshes~\cite{CR1973:2022:CR}. Within the realm of the Crouzeix–Raviart finite element method, the careful selection of approximation spaces plays a crucial role. These spaces are chosen to efficiently represent discontinuous solutions, making the method particularly well-suited for problems with singularities or discontinuities. This distinctive feature has established the Crouzeix--Raviart finite element as a valuable tool across various fields, including fluid dynamics, structural mechanics, and other field where obtaining accurate and flexible numerical solutions to PDEs is essential~\cite{chatzipantelidis1999finite, hansbo2003discontinuous, burman2005stabilized, zhu2014analysis, younes2014combination, di2015extension, ve2019quasi, chen2019finite}.
In this paper, we present two one-parameter families of quadratic polynomial enrichments for the classical Crouzeix--Raviart finite element. These enrichments are obtained by using weighted line integrals as enriched linear functionals and quadratic polynomial functions as enrichment functions.

The paper is organized as follows. In Section~\ref{s1}, we introduce a new one-parameter family of enrichment of the classical Crouzeix--Raviart finite element using quadratic polynomial functions. Additionally, we provide an explicit expression for the associated basis functions and introduce a new quadratic approximation operator based on this enriched finite element. In line with the previous section, in Section~\ref{sec3} we propose another class of enrichment of the Crouzeix--Raviart finite element by employing quadratic polynomial functions. Numerical results are presented in Section~\ref{sec4}.

\section{A first one-parameter family of enrichment of the Crouzeix--Raviart finite element}\label{s1}
Let $T \subset \mathbb{R}^2$ be a nondegenerate triangle with vertices $\B{v}_1, \B{v}_2, \B{v}_3$. Let $\lambda_i$ denote the barycentric coordinate associated to the vertex $\B{v}_i$, and let $\Gamma_i$ represent the edge of $T$ opposite the vertex $\B{v}_i$, $i=1,2,3$. These functions are affine and satisfy the following properties
\begin{equation}\label{bary}
\lambda_i(\B{v}_j) = \delta_{ij},
\qquad i,j=1,2,3,
\end{equation}
and
\begin{equation}\label{prop2}
\lambda_i(t \B{x} + (1-t) \B{y}) = t \lambda_i(\B{x}) + (1-t)\lambda_i(\B{y}), \qquad \B{x}, \B{y}\in T,
\end{equation}
where $\delta_{ij}$ is the Kronecker delta function, that is
\begin{equation*}
\delta_{ij}=
\begin{cases}\begin{aligned}
    &1, \quad \text{if } i=j, \\
    &0, \quad \text{if } i\neq j.
\end{aligned}
\end{cases}
\end{equation*}
By~\eqref{bary} and~\eqref{prop2}, we immediately get
\begin{equation}\label{propstar}
\lambda_i(\B{x})=0, \qquad \B{x}\in \Gamma_i, \qquad i=1,2,3.
\end{equation}
Hereafter, we adopt the standard convention that 
\begin{equation*}
    \B{v}_4=\B{v}_1, \quad \B{v}_5=\B{v}_2
\end{equation*}
and that 
\begin{equation*}
    \lambda_4=\lambda_1, \quad  \lambda_5=\lambda_2.
\end{equation*}
We denote by $\mathcal{I}_j^{\mathrm{CR}}$, $j=1,2,3,$ the following linear functional
\begin{equation*}
 \mathcal{I}_j^{\mathrm{CR}}(f) = \int_{0}^{1}f\left(t\B{v}_{j+1}+ (1-t)\B{v}_{j+2}\right)\, dt, \qquad j=1,2,3.
\end{equation*}
The Crouzeix--Raviart finite element is defined locally as the triple
\begin{equation*}
\left(T,\mathbb{P}_{1}(T), \Sigma_T^{\mathrm{CR}}\right),    
\end{equation*}
 where $\mathbb{P}_{1}(T)$ is the space of all bivariate linear polynomials, that is
 \begin{equation*}
     \mathbb{P}_{1}(T)=\operatorname{span}\{\lambda_1,\lambda_2,\lambda_3\}
 \end{equation*}
 and 
 \begin{equation*}
\Sigma_T^{\mathrm{CR}}=\left\{\mathcal{I}_j^{\mathrm{CR}}\, : \, j=1,2,3\right\}.    
 \end{equation*}
The main goal of this section is to construct a parametric family of quadratic enrichments of the Crouzeix--Raviart finite element. To achieve this, we introduce a real parameter $\alpha > -1$. Using this parameter, we define the following enriched linear functionals
  \begin{equation*}
\mathcal{F}_{j,\alpha}^{{\mathrm{enr}}}(f) = \int_{0}^{1}w_\alpha(t)f\left(t\B v_{j}+ (1-t)
\B m_{j}
\right)\, dt, \qquad j=1,2,3,
\end{equation*} 
 where 
\begin{equation*}
    \B m_{1}=\frac{\B v_{2}+\B v_{3}}{2}\in \Gamma_1, \qquad  \B m_{2}=\frac{\B v_{3}+\B v_{1}}{2}\in \Gamma_2, \qquad  \B m_{3}=\frac{\B v_{1}+\B v_{2}}{2}\in \Gamma_3, 
\end{equation*}
 and
\begin{equation*}
     w_\alpha(t) = t^{\alpha} (1 - t)^{\alpha}
\end{equation*}
is the Jacobi weight function defined on $[0,1]$.
The one-parameter family of triples is then defined as
 \begin{equation}\label{trips}
   \mathcal{C}_{\alpha}= (T,  \mathbb{P}_2(T),\Sigma_{T,\alpha}^{{\mathrm{enr}}}),
 \end{equation}
 where $\mathbb{P}_{2}(T)$ is the space of all bivariate quadratic polynomials, that is
 \begin{equation*}
     \mathbb{P}_{2}(T)=\operatorname{span}\{\lambda_1,\lambda_2,\lambda_3, \lambda_1^2,\lambda_2^2,\lambda_3^2\}
 \end{equation*}   
 and
 \begin{equation*}
\Sigma_{T,\alpha}^{{\mathrm{enr}}}=  \left\{\mathcal{I}_j^{\mathrm{CR}},\mathcal{F}^{\mathrm{enr}}_{j,\alpha}\, : \, j=1,2,3\right\}
\end{equation*}
is the set of enriched degrees of freedom. 
In the following, we prove that for 
\begin{equation*}
    \alpha\neq -\frac{6}{7},
\end{equation*}
 the triple $ \mathcal{C}_{\alpha}$ defined in~\eqref{trips} is a finite element. To establish this, we require some preliminary results.
First and foremost, we unveil explicit formulations for the basis functions of $\mathbb{P}_2(T)$
\begin{equation*}
\mathcal{B}_{AF3}=\{\varphi_i, \phi_i\, : \, i=1,2,3\},
\end{equation*}
designed to satisfy conditions~\cite[Ch. 2]{Davis:1975:IAA}
\begin{equation*}
\mathcal{L}^{\mathrm{enr}}_j(\varphi_i)=\delta_{ij}, \quad  \mathcal{I}_j^{\mathrm{CR}}(\varphi_i)=0, \quad i,j=1,2,3,
\end{equation*}
\begin{equation*}
    \mathcal{L}^{\mathrm{enr}}_j(\phi_i)=0, \quad  \mathcal{I}_j^{\mathrm{CR}}(\phi_i)=\delta_{ij}, \quad i,j=1,2,3,
\end{equation*}
where
\begin{equation}\label{linfunAF3}
\mathcal{L}_j^{\mathrm{enr}}(f)= f(\B{v}_j), \qquad j=1,2,3.
\end{equation}
These functions are recognized as the basis functions of $\mathbb{P}_2(T)$ associated to the enriched finite element~\cite{DellAccio:2022:AUE}
\begin{equation}\label{AF3}
AF3= \left(T,\mathbb{P}_2(T),{\Sigma}_T^{\mathrm{enr}}\right),
\end{equation}
where
\begin{equation*}
{\Sigma}_T^{{\mathrm{enr}}}= \left\{\mathcal{I}_j^{\mathrm{CR}},\mathcal{L}^{\mathrm{enr}}_j\, : \, j=1,2,3\right\}.
\end{equation*}

\begin{theorem}\label{th2}
        The basis functions $\varphi_i,\phi_i$,  $i=1,2,3,$ of $\mathbb{P}_2(T)$ associated to the finite element $AF3$ have the following expressions
\begin{equation}
\varphi_1 = \lambda_1 (1-  3\lambda_2  -3 \lambda_3), \qquad  \varphi_2=\lambda_2 (1-  3\lambda_1  -3 \lambda_3), \qquad \varphi_3 = \lambda_3 (1-  3\lambda_1  -3 \lambda_2),
    \label{basisphi}
\end{equation}
\begin{equation}
\phi_1 = 6 \lambda_2\lambda_3,\qquad 
   \phi_2 = 6 \lambda_1\lambda_3,\qquad
    \phi_3 = 6\lambda_1\lambda_2.
    \label{basisvarphi}
\end{equation}
      \end{theorem}
\begin{remark}\label{remmi}
    We remark that any $p\in \mathbb{P}_2(T)$ can be expressed as
    \begin{equation}\label{proveqnew}
        p=\sum_{k=1}^3  \mathcal{L}^{\mathrm{enr}}_k(p)\varphi_k+\sum_{k=1}^3  \mathcal{I}^{\mathrm{CR}}_k(p)\phi_k.
    \end{equation}
\end{remark}

\begin{remark}
        We observe that, as a consequence of~\eqref{prop2} and~\eqref{propstar}, we obtain
 \begin{equation}\label{propimp2}
     \lambda_i(\B m_j)=\frac{1}{2}(1-\delta_{ij}), \qquad i,j=1,2,3.
 \end{equation} 
\end{remark}

The next lemma plays a crucial role in deriving the main result of this section. To aid in this effort, we introduce the following notations
\begin{equation}\label{ut}
     \gamma_{\alpha} =  B(\alpha+2,\alpha+1),   \quad 
   h_{\alpha}=-\frac{5\alpha+6}{4(2\alpha+3)},\quad 
   K_{\alpha}=-\frac{\alpha}{2\alpha+3},
\end{equation}
 where 
\begin{equation}\label{beta}
    B(z_1,z_2)=\int_{0}^{1}u^{z_1-1}(1-u)^{z_2-1} du, \qquad z_1,z_2>-1,
\end{equation}
is the classical Euler beta function~\cite{Abramowitz:1948:HOM}. This function is symmetric, i.e. 
\begin{equation}\label{simmbeta}
    B(z_1,z_2)=B(z_2,z_1), \qquad z_1,z_2>-1,
\end{equation}
and satisfies
\begin{equation}\label{propbetafun}
B(z_1+1,z_2)=\frac{z_1}{z_1+z_2}B(z_1,z_2),  \qquad z_1,z_2>-1.  
\end{equation}

\begin{lemma}\label{lem1}
  Let $p \in \mathbb{P}_2(T)$ such that $ \mathcal{I}_j^{\mathrm{CR}}(p)=0$, $j=1,2,3$. Then, the values of $p$ at the vertices of $T$ satisfy the following linear system \begin{equation}\label{dnnexprima}
  \gamma_{\alpha}  \begin{bmatrix}
 K_{\alpha} &  h_{\alpha} & h_{\alpha}\\
h_{\alpha} & K_{\alpha} & h_{\alpha} \\
h_{\alpha} &h_{\alpha} &  K_{\alpha}
\end{bmatrix}
 \begin{bmatrix}
p(\B v_1)\\
p(\B v_2)\\
p(\B v_3)
\end{bmatrix}
= \begin{bmatrix}
\mathcal{F}_{1,\alpha}^{{\mathrm{enr}}}(p)\\
\mathcal{F}_{2,\alpha}^{{\mathrm{enr}}}(p)\\
\mathcal{F}_{3,\alpha}^{{\mathrm{enr}}}(p)
\end{bmatrix}.
\end{equation}
  
\end{lemma}
\begin{proof}
 Let $p \in \mathbb{P}_2(T)$ such that
 \begin{equation*}
\mathcal{I}_j^{\mathrm{CR}}(p)=0, \qquad j=1,2,3.
 \end{equation*}
Referring to~\eqref{proveqnew} and~\eqref{linfunAF3}, we can express $p$ as follows 
 \begin{equation}\label{eqpb11}
     p=\sum_{k=1}^{3} \mathcal{L}^{\mathrm{enr}}_k(p)\varphi_k+\sum_{k=1}^3 \mathcal{I}^{\mathrm{CR}}_k(p) \phi_k= \sum_{k=1}^{3} p(\B v_k)\varphi_k.
 \end{equation}
Applying the functional $\mathcal{F}_{j,\alpha}^{\mathrm{enr}}$, $j=1,2,3$, to both sides of~\eqref{eqpb11}, we obtain 
\begin{equation}\label{eqpb11111111}
\mathcal{F}_{j,\alpha}^{\mathrm{enr}}(p) = \sum_{k=1}^{3} p(\B{v}_k)\mathcal{F}_{j,\alpha}^{\mathrm{enr}}(\varphi_k), \qquad j=1,2,3.
\end{equation}
Using~\eqref{basisphi} in combination with~\eqref{prop2} and~\eqref{propimp2}, and leveraging the properties of the Euler beta function~\eqref{simmbeta} and~\eqref{propbetafun}, straightforward calculations lead to the following results
\begin{eqnarray}
\label{eqimp}
\mathcal{F}_{j,\alpha}^{\mathrm{enr}}(\varphi_j) &=& K_{\alpha}\gamma_{\alpha}, \qquad j=1,2,3,
\\
\mathcal{F}_{j,\alpha}^{\mathrm{enr}}(\varphi_k) &=& h_{\alpha}\gamma_{\alpha}, \qquad j,k=1,2,3, \quad j\neq k.
\label{eqimp2}
\end{eqnarray}
By substituting~\eqref{eqimp} and~\eqref{eqimp2} into~\eqref{eqpb11111111}, the thesis follows.
\end{proof}

\begin{lemma}\label{lem2}
  Let $\alpha > -1$ be a fixed real number different from $-\frac{6}{7}$ and let $p\in \mathbb{P}_2(T)$ such that $\mathcal{I}_j^{\mathrm{CR}}(p)=0$, $j=1,2,3$. Then $p$ vanishes at the vertices of the triangle $T$ if and only if
\begin{equation}\label{dim}
\mathcal{F}_{j,\alpha}^{{\mathrm{enr}}}(p)= 0, \quad j=1,2,3.
\end{equation}
 \end{lemma}
 \begin{proof} Let $p\in \mathbb{P}_2(T)$ such that 
   \begin{equation*}
  \mathcal{I}_j^{\mathrm{CR}}(p) = 0, \qquad j=1,2,3.
\end{equation*}
By leveraging~\eqref{dnnexprima} of Lemma~\ref{lem1}, it follows that, if $p$ vanishes at the vertices of the triangle $T$, then condition~\eqref{dim} is satisfied. It remains to prove the reverse implication. Then we assume that $p$ satisfies condition~\eqref{dim}. In this hypothesis, the linear system~\eqref{dnnexprima} becomes
 \begin{equation*}
  \gamma_{\alpha}  \begin{bmatrix}
 K_{\alpha} &  h_{\alpha} & h_{\alpha}\\
h_{\alpha} & K_{\alpha} & h_{\alpha} \\
h_{\alpha} &h_{\alpha} &  K_{\alpha}
\end{bmatrix}
 \begin{bmatrix}
p(\B v_1)\\
p(\B v_2)\\
p(\B v_3)
\end{bmatrix}
= \begin{bmatrix}
 0\\
 0\\
0
\end{bmatrix},
 \end{equation*}
where $\gamma_{\alpha},h_{\alpha}, K_{\alpha}$ are defined in~\eqref{ut}.
The required result follows since the determinant of the matrix associated to the above linear system is 
 \begin{equation*}
 \Delta_{\alpha}=- \gamma_{\alpha}^3\frac{(6 + \alpha)^2 (6 + 7 \alpha)}{32(3 + 2 \alpha)^3},  
 \end{equation*}
 which is different from zero for any $\alpha>-1$ and $\alpha\neq -\frac{6}{7}$. 
 \end{proof}\\

 \begin{theorem}\label{th1} 
  For any real number $\alpha > -1$, and different from $-\frac{6}{7}$, the triple $ \mathcal{C}_{\alpha}$ is a finite element.
\end{theorem}
\begin{proof}
Let $p\in \mathbb{P}_2(T)$ such that
\begin{eqnarray}
  \mathcal{I}_j^{\mathrm{CR}}(p) &=& 0, \qquad j=1,2,3, \label{ho1}\\
  \mathcal{F}_{j,\alpha}^{\mathrm{enr}}(p) &=& 0, \qquad j=1,2,3. \label{ho21}
\end{eqnarray}
We must prove that $p=0$. 
By Lemma~\ref{lem2}, we deduce that conditions~\eqref{ho1} and~\eqref{ho21} are equivalent to
 \begin{eqnarray*}
  \mathcal{I}_j^{\mathrm{CR}}(p) &=& 0, \qquad j=1,2,3, \\
 \mathcal{L}^{\mathrm{enr}}_j(p) &=& 0, \qquad j=1,2,3, 
\end{eqnarray*} 
where $\mathcal{L}^{\mathrm{enr}}_j$, $j=1,2,3,$ is defined in~\eqref{linfunAF3}. 
The desired result can be directly inferred from the fact that the triple $AF3$, defined in~\eqref{AF3}, is a finite element, as established in~\cite[Theorem 2.6]{DellAccio:2022:AUE}.
\end{proof}

As a direct consequence of Theorem~\ref{th1}, we can establish the existence of a basis
\begin{equation*}
    \mathcal{B}_{\mathcal{C}_{\alpha}}=\{\psi_i, \zeta_i\, :\, i=1,2,3\}
\end{equation*}
of $\mathbb{P}_2(T)$ satisfying 
the following conditions~\cite[Ch. 2]{Davis:1975:IAA}
\begin{eqnarray}
\label{propvarphi}
&& \mathcal{I}_j^{\mathrm{CR}}(\psi_i) = \delta_{ij}, \quad \mathcal{F}^{\mathrm{enr}}_{j,\alpha}(\psi_i) = 0, \qquad i,j=1,2,3, \\ \label{propphi}
&&\mathcal{I}_j^{\mathrm{CR}}(\zeta_i) = 0, \quad \mathcal{F}^{\mathrm{enr}}_{j,\alpha}(\zeta_i) = \delta_{ij}, \qquad i,j=1,2,3.
\end{eqnarray}
These functions, known as the basis functions of $\mathbb{P}_2(T)$ associated to the enriched finite element $\mathcal{C}_{\alpha}$, are explicitly expressed in the following theorem.
\begin{theorem}\label{th2allalf} 
Let $\alpha>-1$ be a fixed real number different from $-\frac{6}{7}$. The basis functions $\psi_i,\zeta_i$,  $i=1,2,3,$ of $\mathbb{P}_2(T)$ associated to the finite element $\mathcal{C}_{\alpha}$ have the following expressions
\begin{equation}
    \psi_i=c_{\alpha} \varphi_i +d_{\alpha}\sum_{\substack{k=1 \\ k \neq i}}^{3} \varphi_k + \phi_i,  \qquad
    \zeta_i=\frac{1}{\gamma_{\alpha} L_{\alpha}}\left((K_{\alpha}+h_{\alpha})\varphi_i-h_{\alpha} \sum_{\substack{k=1 \\ k \neq i}}^{3} \varphi_k \right), \qquad i=1,2,3,
\end{equation}
where
\begin{equation}\label{not}
c_{\alpha}=\frac{3 (11 \alpha^2  + 20 \alpha +12 )}{(\alpha+6) ( 7 \alpha+6)}, \qquad d_{\alpha}=\frac{3(- 3 \alpha^2  + 8 \alpha +12 ) }{(\alpha+ 6 ) (7 \alpha+6 )}, \qquad 
    L_{\alpha}=-\frac{(\alpha+ 6 ) (7 \alpha+6 )}{8 (2 \alpha+3)^2},
\end{equation}
and $\gamma_{\alpha},h_{\alpha},K_{\alpha},\varphi_i,\phi_i$ are defined in~\eqref{ut}~\eqref{basisphi} and~\eqref{basisvarphi}, respectively. 
      \end{theorem}
\begin{proof}
   Without loss on generality, we prove the theorem for $i=1$. 

Since $\psi_1\in\mathbb{P}_2(T)$, according to~\eqref{proveqnew}, it can be expressed as
\begin{equation*}
    \psi_1=\sum_{k=1}^3 \mathcal{L}^{\mathrm{enr}}_k(\psi_1) \varphi_k+\sum_{k=1}^3 \mathcal{I}^{CR}_k(\psi_1) \phi_k.
\end{equation*}
Using~\eqref{propvarphi} and~\eqref{linfunAF3}, we obtain
\begin{equation} \label{djeie}
    \psi_1=\sum_{k=1}^3 \psi_1(\B v_k) \varphi_k+\phi_1.
\end{equation}
Now, let us compute the vector
\begin{equation*}
    \left[\psi_1(\B v_1),\psi_1(\B v_2),\psi_1(\B v_3)\right]^T. 
\end{equation*}
For this purpose, we apply the operators $\mathcal{F}^{\mathrm{enr}}_{j,\alpha}$, $j=1,2,3$, to both sides of~\eqref{djeie}, leveraging~\eqref{propvarphi}, resulting in the following 
    \begin{eqnarray*}
0  &=& \psi_1(\B v_1) \mathcal{F}_{1,\alpha}^{{\mathrm{enr}}}(\varphi_1) + \psi_1(\B v_2)\mathcal{F}_{1,\alpha}^{{\mathrm{enr}}}(\varphi_2)+  \psi_1(\B v_3)\mathcal{F}_{1,\alpha}^{{\mathrm{enr}}}(\varphi_3)+\mathcal{F}_{1,\alpha}^{{\mathrm{enr}}}(\phi_1) \\
      0&=&  \psi_1(\B v_1)  \mathcal{F}_{2,\alpha}^{{\mathrm{enr}}}(\varphi_1)+\psi_1(\B v_2)\mathcal{F}_{2,\alpha}^{{\mathrm{enr}}}(\varphi_2)+ \psi_1(\B v_3) \mathcal{F}_{2,\alpha}^{{\mathrm{enr}}}(\varphi_3)+\mathcal{F}_{2,\alpha}^{{\mathrm{enr}}}(\phi_1)  \\
       0 &=&  \psi_1(\B v_1) \mathcal{F}_{3,\alpha}^{{\mathrm{enr}}}(\varphi_1) + \psi_1(\B v_2) \mathcal{F}_{3,\alpha}^{{\mathrm{enr}}}(\varphi_2) +\psi_1(\B v_3)\mathcal{F}_{3,\alpha}^{{\mathrm{enr}}}(\varphi_3)+\mathcal{F}_{3,\alpha}^{{\mathrm{enr}}}(\phi_1). 
  \end{eqnarray*}
Substituting~\eqref{eqimp} and~\eqref{eqimp2} into these equations, we obtain the linear system
\begin{equation}\label{mdiendekm}
  \gamma_{\alpha}  \begin{bmatrix}
 K_{\alpha} &  h_{\alpha} & h_{\alpha}\\
h_{\alpha} & K_{\alpha} & h_{\alpha} \\
h_{\alpha} &h_{\alpha} &  K_{\alpha}
\end{bmatrix}
 \begin{bmatrix}
\psi_1(\B v_1)\\
\psi_1(\B v_2)\\
\psi_1(\B v_3)
\end{bmatrix}
= -\begin{bmatrix}
\mathcal{F}_{1,\alpha}^{{\mathrm{enr}}}(\phi_1)\\
\mathcal{F}_{2,\alpha}^{{\mathrm{enr}}}(\phi_1)\\
\mathcal{F}_{3,\alpha}^{{\mathrm{enr}}}(\phi_1)
\end{bmatrix}.
\end{equation}
Thus, performing a straightforward calculation, we find
 \begin{equation*}
  -\frac{1}{\gamma_{\alpha}L_{\alpha}}  \begin{bmatrix}
 K_{\alpha}+h_{\alpha} &  -h_{\alpha} & -h_{\alpha}\\
-h_{\alpha} & K_{\alpha}+h_{\alpha} & -h_{\alpha} \\
-h_{\alpha} & -h_{\alpha} &   K_{\alpha}+h_{\alpha}
\end{bmatrix}
\begin{bmatrix}
\mathcal{F}_{1,\alpha}^{{\mathrm{enr}}}(\phi_1)\\
\mathcal{F}_{2,\alpha}^{{\mathrm{enr}}}(\phi_1)\\
\mathcal{F}_{3,\alpha}^{{\mathrm{enr}}}(\phi_1)
\end{bmatrix}
=  \begin{bmatrix}
\psi_1(\B v_1)\\
\psi_1(\B v_2)\\
\psi_1(\B v_3)
\end{bmatrix}.
\end{equation*}
By~\eqref{bary},~\eqref{prop2} and~\eqref{propimp2}, we get 
\begin{equation}\label{newnewnew}
    \mathcal{F}^{\mathrm{enr}}_{1,\alpha}(\phi_1)=\frac{3(\alpha+2)}{2(2\alpha+3)}\gamma_{\alpha}, \qquad   \mathcal{F}^{\mathrm{enr}}_{2,\alpha}(\phi_1)= \mathcal{F}^{\mathrm{enr}}_{3,\alpha}(\phi_1)=\frac{3(\alpha+1)}{2\alpha+3}\gamma_{\alpha}.
\end{equation}
Substituting~\eqref{newnewnew} into the linear system~\eqref{mdiendekm} and solving it, we obtain 
\begin{equation}\label{snjdb}
    \psi_1(\B v_1)=\frac{3 (11 \alpha^2  + 20 \alpha +12 )}{(\alpha+6) ( 7 \alpha+6)}, \qquad   \psi_1(\B v_2)=\psi_1(\B v_3)=\frac{3(- 3 \alpha^2  + 8 \alpha +12 ) }{(\alpha+ 6 ) (7 \alpha+6 )}.
\end{equation}
Combining~\eqref{snjdb} with~\eqref{djeie} and following the notations defined in~\eqref{not}, we find
\begin{equation*}
\psi_1=c_{\alpha}\varphi_1+d_{\alpha}\left(\varphi_2+\varphi_3\right) + \phi_1.
\end{equation*}

It remains to prove the expression for $\zeta_1$. Since $\zeta_1\in \mathbb{P}_2(T)$, by~\eqref{proveqnew}, it can be expressed as 
\begin{equation*}
    \zeta_1=\sum_{k=1}^3 \mathcal{L}^{\mathrm{enr}}_k(\zeta_1) \varphi_k+\sum_{k=1}^3 \mathcal{I}^{CR}_k(\zeta_1) \phi_k.
\end{equation*}
Using~\eqref{propphi} and~\eqref{linfunAF3},  we obtain
\begin{equation} \label{die}
    \zeta_1=\sum_{k=1}^3 \zeta_1(\B v_j) \varphi_k.
\end{equation}
Now, let us compute the vector
\begin{equation*}
    \left[\zeta_1(\B v_1),\zeta_1(\B v_2),\zeta_1(\B v_3)\right]^T. 
\end{equation*}
For this purpose, leveraging~\eqref{dnnexprima} of Lemma~\ref{lem1} and~\eqref{propphi}, we get
 \begin{equation*}
  \gamma_{\alpha}  \begin{bmatrix}
 K_{\alpha} &  h_{\alpha} & h_{\alpha}\\
h_{\alpha} & K_{\alpha} & h_{\alpha} \\
h_{\alpha} &h_{\alpha} &  K_{\alpha}
\end{bmatrix}
 \begin{bmatrix}
\zeta_1(\B v_1)\\
\zeta_1(\B v_2)\\
\zeta_1(\B v_3)
\end{bmatrix}
= \begin{bmatrix}
1\\
0\\
0
\end{bmatrix}.
\end{equation*}
Thus, we find
 \begin{equation*}
  \frac{1}{\gamma_{\alpha}L_{\alpha}}  \begin{bmatrix}
 K_{\alpha}+h_{\alpha} &  -h_{\alpha} & -h_{\alpha}\\
-h_{\alpha} & K_{\alpha}+h_{\alpha} & -h_{\alpha} \\
-h_{\alpha} & -h_{\alpha} &   K_{\alpha}+h_{\alpha}
\end{bmatrix}
 \begin{bmatrix}
1\\
0\\
0
\end{bmatrix}
=  \begin{bmatrix}
\zeta_1(\B v_1)\\
\zeta_1(\B v_2)\\
\zeta_1(\B v_3)
\end{bmatrix},
\end{equation*}
and then
\begin{equation*}
    \zeta_1(\B v_1)=\frac{K_{\alpha}+h_{\alpha}}{\gamma_{\alpha} L_{\alpha}}, \qquad  \zeta_1(\B v_2)=-\frac{h_{\alpha}}{\gamma_{\alpha} L_{\alpha}}, \qquad \zeta_1(\B v_3)=-\frac{h_{\alpha}}{\gamma_{\alpha} L_{\alpha}}.
\end{equation*}
Substituting these values into~\eqref{die}, we have 
\begin{equation*}
    \zeta_1=\frac{1}{\gamma_{\alpha}L_{\alpha}}\left[\left(K_{\alpha}+h_{\alpha}\right)\varphi_1-h_{\alpha} \varphi_2-h_{\alpha}\varphi_3\right].
\end{equation*}
Analogously, the theorem can be established for $i=2$ and $i=3$; consequently, the thesis follows.
\end{proof}

\begin{theorem}
 Let $\alpha>-1$ be a fixed real number different from $-\frac{6}{7}$. The approximation operator relative to the enriched finite element $\mathcal{C}_{\alpha}$
\begin{equation}
\begin{array}{rcl}
{\Pi}_{2,\mathcal{C}_\alpha}^{{\mathrm{enr}}}: C(T) &\rightarrow& \mathbb{P}_2(T)
\\
f &\mapsto& \displaystyle{\sum_{j=1}^{3}  \mathcal{I}^{\mathrm{CR}}_j(f)\psi_j+ \mathcal{F}^{\mathrm{enr}}_{j,\alpha}(f)}\zeta_j,
\end{array}
\label{pilinch9C}
\end{equation}
reproduces all polynomials of $\mathbb{P}_2(T)$ and satisfies 
\begin{eqnarray*}
\mathcal{I}^{\mathrm{CR}}_j\left({\Pi}_{2,\mathcal{C}_\alpha}^{\mathrm{enr}}[f]\right)&=&\mathcal{I}^{\mathrm{CR}}_j(f), \qquad j=1,2,3, \\ \mathcal{F}^{\mathrm{enr}}_{j,\alpha}\left({\Pi}_{2,\mathcal{C}_\alpha}^{\mathrm{enr}}[f]\right)&=&\mathcal{F}^{\mathrm{enr}}_{j,\alpha}(f), \qquad j=1,2,3.
\end{eqnarray*} 
\end{theorem}
\begin{proof}
The proof is a consequence of~\eqref{propvarphi} and~\eqref{propphi}.
\end{proof}

\section{A second one-parameter family of enrichment of the Crouzeix--Raviart finite element}\label{sec3}

In this section, we introduce a new parametric family of quadratic enrichments of the Crouzeix--Raviart finite element. To achieve this, we introduce a real parameter $\beta > -1$. Using this parameter, we define the following enriched linear functionals 
 \begin{equation}\label{ex3a}
\mathcal{G}_{j,\beta}^{{\mathrm{enr}}}(f) = \int_{0}^{1}w_{\beta}(t)f\left(t\B v_{j}+ (1-t)
 \B m^{\star}
\right)\, dt, \qquad j=1,2,3,
\end{equation} 
 where 
 \begin{equation*}
     \B m^{\star}=\frac{\B v_1+\B v_2+\B v_3}{3}
 \end{equation*}
is the center of gravity of $T$ and $w_\beta(t)=t^\beta(1-t)^{1-\beta}$.
The one-parameter family of triples is then defined as
 \begin{equation}\label{trisps}
   \mathcal{E}_{\beta}= (T,  \mathbb{P}_2(T),\Sigma_{T,\beta}^{{\mathrm{enr}}}),
 \end{equation}
where   
 \begin{equation}\label{pasf2}
 \Sigma_{T,\beta}^{{\mathrm{enr}}}=  \left\{\mathcal{I}_j^{\mathrm{CR}},\mathcal{G}^{\mathrm{enr}}_{j,\beta}\, : \, j=1,2,3\right\},
\end{equation}
is the set of enriched degrees of freedom.  In the following, we establish that for any $\beta>-1$, the triple $ \mathcal{E}_{\beta}$, defined in~\eqref{trisps}, is a finite element. To prove this, we first introduce some preliminary results. By using the same notations of the previous section, we denote by
\begin{equation}\label{nu}
\nu_{\beta} = \frac{B(\beta+2,\beta+1)}{3(2\beta+3)}, \qquad \sigma_{\beta}=3\beta+4.
\end{equation}
\begin{remark}
        We observe that, as a consequence of~\eqref{prop2} and~\eqref{propstar}, we obtain
 \begin{equation}\label{propifmp2}
     \lambda_i(\B m^{\star})=\frac{1}{3}, \qquad i=1,2,3.
 \end{equation} 
\end{remark}
 
\begin{lemma}\label{lem12}
  Let $p \in \mathbb{P}_2(T)$ such that $ \mathcal{I}_j^{\mathrm{CR}}(p)=0$, $j=1,2,3$. Then, the values of $p$ at the vertices of $T$ satisfy the following linear system 
\begin{equation}\label{dnnexprimva}
  \nu_{\beta}  \begin{bmatrix}
2 &  -\sigma_{\beta} & -\sigma_{\beta}\\
-\sigma_{\beta} & 2 & -\sigma_{\beta} \\
-\sigma_{\beta} &-\sigma_{\beta}&  2
\end{bmatrix}
 \begin{bmatrix}
p(\B v_1)\\
p(\B v_2)\\
p(\B v_3)
\end{bmatrix}
= \begin{bmatrix}
\mathcal{G}_{1,\beta}^{{\mathrm{enr}}}(p)\\
\mathcal{G}_{2,\beta}^{{\mathrm{enr}}}(p)\\
\mathcal{G}_{3,\beta}^{{\mathrm{enr}}}(p)
\end{bmatrix}.
\end{equation}
  
\end{lemma}
\begin{proof}
 Let $p \in \mathbb{P}_2(T)$ such that
 \begin{equation*}
     \mathcal{I}_j^{\mathrm{CR}}(p)=0, \qquad j=1,2,3.
 \end{equation*}
Referring to~\eqref{proveqnew} and~\eqref{linfunAF3}, we can express $p$ as follows 
 \begin{equation}\label{eqcpb11}
   p=\sum_{k=1}^{3} \mathcal{L}^{\mathrm{enr}}_k(p)\varphi_k+\sum_{k=1}^3 \mathcal{I}^{\mathrm{CR}}_k(p) \phi_k= \sum_{k=1}^{3} p(\B v_k)\varphi_k.
 \end{equation}
Applying the functional $\mathcal{G}_{j,\beta}^{\mathrm{enr}}$, $j=1,2,3$, to both sides of~\eqref{eqcpb11}, we obtain 
\begin{equation}\label{eqpb1111111}
\mathcal{G}_{j,\beta}^{\mathrm{enr}}(p) = \sum_{k=1}^{3} p(\B{v}_k)\mathcal{G}_{j,\beta}^{\mathrm{enr}}(\varphi_k), \qquad j=1,2,3.
\end{equation}
Using~\eqref{basisphi} in combination with~\eqref{prop2} and~\eqref{propimp2}, and leveraging the properties of the Euler beta function~\eqref{simmbeta} and~\eqref{propbetafun}, straightforward calculations lead to the following results
\begin{eqnarray}
\label{aiutcomplet}
\mathcal{G}^{\mathrm{enr}}_{j,\beta}(\varphi_j)&=&2\nu_{\beta}, \qquad j=1,2,3,
\\
 \mathcal{G}^{\mathrm{enr}}_{j,\beta}(\varphi_k)&=&-\sigma_{\beta}\nu_{\beta}, \qquad j,k=1,2,3, \quad j\neq k.
\label{aiutcomplet1}
\end{eqnarray}
By substituting~\eqref{aiutcomplet} and~\eqref{aiutcomplet1} into~\eqref{eqpb1111111}, the thesis follows.

\end{proof}

\begin{lemma}\label{lem22}
  Let $p\in \mathbb{P}_2(T)$ such that $\mathcal{I}_j^{\mathrm{CR}}(p)=0$, $j=1,2,3$. Then $p$ vanishes at the vertices of the triangle $T$ if and only if
\begin{equation}\label{dims}
\mathcal{G}_{j,\beta}^{{\mathrm{enr}}}(p)= 0, \quad j=1,2,3.
\end{equation}
 \end{lemma}
 \begin{proof} 
  Let $p\in \mathbb{P}_2(T)$ such that 
   \begin{equation*}
  \mathcal{I}_j^{\mathrm{CR}}(p) = 0, \qquad j=1,2,3.
\end{equation*}
 By leveraging~\eqref{dnnexprimva} of Lemma~\ref{lem12}, it follows that, if $p$ vanishes at the vertices of the triangle $T$, then condition~\eqref{dims} is satisfied. 
 It remains to prove the reverse implication. Then we assume that $p$ satisfies condition~\eqref{dims}. In this hypothesis, the linear system~\eqref{dnnexprimva} becomes
 \begin{equation*}
  \nu_{\beta}  \begin{bmatrix}
 2 &  -\sigma_{\beta} & -\sigma_{\beta}\\
-\sigma_{\beta} & 2 & -\sigma_{\beta} \\
-\sigma_{\beta} &-\sigma_{\beta} &  2
\end{bmatrix}
 \begin{bmatrix}
p(\B v_1)\\
p(\B v_2)\\
p(\B v_3)
\end{bmatrix}
= \begin{bmatrix}
 0\\
 0\\
0
\end{bmatrix},
 \end{equation*}
where $\nu_{\beta},\sigma_{\beta}$ are defined in~\eqref{nu}.
The required result follows since the determinant of the matrix associated to the above linear system is 
\begin{equation*}
\Delta_{\beta}=-54\nu_{\beta}^3 (\beta+1)(\beta+2)^2
\end{equation*}
 which is different from zero for any $\beta>-1$.
 \end{proof}\\

 \begin{theorem}\label{th22} 
 For any real number $\beta > -1$ the triple $ \mathcal{E}_{\beta}$ is a finite element. 
\end{theorem}
\begin{proof}
  Let $p\in \mathbb{P}_2(T)$ such that
\begin{eqnarray}
  \mathcal{I}_j^{\mathrm{CR}}(p) &=& 0, \qquad j=1,2,3, \label{ho1new}\\
  \mathcal{G}_{j,\beta}^{\mathrm{enr}}(p) &=& 0, \qquad j=1,2,3. \label{ho21new}
\end{eqnarray}
We must prove that $p=0$. 
By Lemma~\ref{lem22}, we deduce that conditions~\eqref{ho1new} and~\eqref{ho21new} are equivalent to
 \begin{eqnarray*}
  \mathcal{I}_j^{\mathrm{CR}}(p) &=& 0, \qquad j=1,2,3, \\
 \mathcal{L}^{\mathrm{enr}}_j(p) &=& 0, \qquad j=1,2,3, 
\end{eqnarray*} 
where $\mathcal{L}^{\mathrm{enr}}_j$, $j=1,2,3,$ is defined in~\eqref{linfunAF3}. 
The desired result can be directly inferred from the fact that the triple $AF3$, defined in~\eqref{AF3}, is a finite element, as established in~\cite[Theorem 2.6]{DellAccio:2022:AUE}.
\end{proof}

As a direct consequence of Theorem~\ref{th22}, we can establish the existence of a basis
\begin{equation*}
    \mathcal{B}_{\mathcal{E}_{\beta}}=\{\tau_i, \rho_i\, :\, i=1,2,3\}
\end{equation*}
of $\mathbb{P}_2(T)$ satisfying 
the following conditions~\cite[Ch. 2]{Davis:1975:IAA}
\begin{eqnarray}
\label{proptau}
&& \mathcal{I}_j^{\mathrm{CR}}(\tau_i) = \delta_{ij}, \quad \mathcal{G}^{\mathrm{enr}}_{j,\beta}(\tau_i) = 0, \qquad i,j=1,2,3, \\ \label{propro}
&&\mathcal{I}_j^{\mathrm{CR}}(\rho_i) = 0, \quad \mathcal{G}^{\mathrm{enr}}_{j,\beta}(\rho_i) = \delta_{ij}, \qquad i,j=1,2,3.
\end{eqnarray}
These functions, known as the basis functions of $\mathbb{P}_2(T)$ associated to the enriched finite element $\mathcal{E}_{\beta}$, are explicitly expressed in the following theorem.

\begin{theorem}
Let $\beta>-1$ be a fixed real number. The basis functions $\tau_i,\rho_i$,  $i=1,2,3,$ of $\mathbb{P}_2(T)$ associated to the finite element $\mathcal{E}_{\beta}$ have the following expressions
\begin{equation}
    \tau_i=r_{\beta}\varphi_i+m_{\beta}\sum_{\substack{k=1 \\ k \neq i}}^{3} \varphi_k+\phi_i,  \qquad
    \rho_i= \frac{1}{\nu_{\beta}L_{\beta}}\left(\left(\sigma_{\beta}-2\right)\varphi_i-\sigma_{\beta} \sum_{\substack{k=1 \\ k \neq i}}^{3} \varphi_k \right), \qquad i=1,2,3,
\end{equation}
where 
\begin{equation}\label{not1}
     r_{\beta}=\frac{6(7 \beta^2 + 18 \beta + 12)}{ L_{\beta}}, \qquad   m_{\beta}=\frac{6(\beta^2 + 6 \beta + 6)}{ L_{\beta}}, \qquad
  L_{\beta}=18(\beta+1)(\beta+2)
\end{equation}
and $\nu_{\beta}, \sigma_{\beta},\varphi_i,\phi_i$ are defined in~\eqref{nu}~\eqref{basisphi} and~\eqref{basisvarphi}, respectively. 
      \end{theorem}
\begin{proof}
   Without loss on generality, we prove the theorem for $i=1$. 

Since $\tau_1\in\mathbb{P}_2(T)$, according to~\eqref{proveqnew}, it can be expressed as
\begin{equation*}
    \tau_1=\sum_{k=1}^3 \mathcal{L}^{\mathrm{enr}}_k(\tau_1) \varphi_k+\sum_{k=1}^3 \mathcal{I}^{CR}_k(\tau_1) \phi_k.
\end{equation*}
Using~\eqref{proptau} and~\eqref{linfunAF3}, we obtain
\begin{equation} \label{djeise}
    \tau_1=\sum_{k=1}^3 \tau_1(\B v_k) \varphi_k+\phi_1.
\end{equation}
Now, let us compute the vector
\begin{equation*}
    \left[\tau_1(\B v_1),\tau_1(\B v_2),\tau_1(\B v_3)\right]^T. 
\end{equation*}
For this purpose, we apply the operators $\mathcal{G}^{\mathrm{enr}}_{j,\beta}$, $j=1,2,3$, to both sides of~\eqref{djeise}, leveraging~\eqref{proptau}, resulting in the following 
    \begin{eqnarray*}
0  &=& \tau_1(\B v_1) \mathcal{G}_{1,\beta}^{{\mathrm{enr}}}(\varphi_1) + \tau_1(\B v_2)\mathcal{G}_{1,\beta}^{{\mathrm{enr}}}(\varphi_2)+  \tau_1(\B v_3)\mathcal{G}_{1,\beta}^{{\mathrm{enr}}}(\varphi_3)+\mathcal{G}_{1,\beta}^{{\mathrm{enr}}}(\phi_1) \\
     0  &=& \tau_1(\B v_1) \mathcal{G}_{2,\beta}^{{\mathrm{enr}}}(\varphi_1) + \tau_1(\B v_2)\mathcal{G}_{2,\beta}^{{\mathrm{enr}}}(\varphi_2)+  \tau_1(\B v_3)\mathcal{G}_{2,\beta}^{{\mathrm{enr}}}(\varphi_3)+\mathcal{G}_{2,\beta}^{{\mathrm{enr}}}(\phi_1)\\
       0  &=& \tau_1(\B v_1) \mathcal{G}_{3,\beta}^{{\mathrm{enr}}}(\varphi_1) + \tau_1(\B v_2)\mathcal{G}_{3,\beta}^{{\mathrm{enr}}}(\varphi_2)+  \tau_1(\B v_3)\mathcal{G}_{3,\beta}^{{\mathrm{enr}}}(\varphi_3)+\mathcal{G}_{3,\beta}^{{\mathrm{enr}}}(\phi_1). 
  \end{eqnarray*}
Substituting~\eqref{aiutcomplet} and~\eqref{aiutcomplet1} into these equations, we obtain the linear system

\begin{equation}\label{mdiendekm1}
   \nu_{\beta}  \begin{bmatrix}
2 &  -\sigma_{\beta} & -\sigma_{\beta}\\
-\sigma_{\beta} & 2 & -\sigma_{\beta} \\
-\sigma_{\beta} &-\sigma_{\beta}&  2
\end{bmatrix}
 \begin{bmatrix}
\tau_1(\B v_1)\\
\tau_1(\B v_2)\\
\tau_1(\B v_3)
\end{bmatrix}
= -\begin{bmatrix}
\mathcal{G}_{1,\beta}^{{\mathrm{enr}}}(\phi_1)\\
\mathcal{G}_{2,\beta}^{{\mathrm{enr}}}(\phi_1)\\
\mathcal{G}_{3,\beta}^{{\mathrm{enr}}}(\phi_1)
\end{bmatrix}.
\end{equation}
Thus, performing a straightforward calculation, we find
 \begin{equation*}
 - \frac{1}{\nu_{\beta}L_{\beta}}  \begin{bmatrix}
 \sigma_{\beta}-2 &  -\sigma_{\beta} & -\sigma_{\beta}\\
-\sigma_{\beta} & \sigma_{\beta}-2 & -\sigma_{\beta} \\
-\sigma_{\beta} & -\sigma_{\beta} &   \sigma_{\beta}-2
\end{bmatrix}
\begin{bmatrix}
\mathcal{G}_{1,\beta}^{{\mathrm{enr}}}(\phi_1)\\
\mathcal{G}_{2,\beta}^{{\mathrm{enr}}}(\phi_1)\\
\mathcal{G}_{3,\beta}^{{\mathrm{enr}}}(\phi_1)
\end{bmatrix}
=  \begin{bmatrix}
\tau_1(\B v_1)\\
\tau_1(\B v_2)\\
\tau_1(\B v_3)
\end{bmatrix}.
\end{equation*}
By~\eqref{bary},~\eqref{prop2} and~\eqref{propimp2}, we get 
\begin{equation}\label{newnewnew1}
    \mathcal{G}^{\mathrm{enr}}_{1,\beta}(\phi_1)=2(\beta+2) \nu_{\beta}, \qquad   \mathcal{G}^{\mathrm{enr}}_{2,\beta}(\phi_1)= \mathcal{G}^{\mathrm{enr}}_{3,\beta}(\phi_1)=2(4\beta+5)\nu_{\beta}.
\end{equation}
Substituting~\eqref{newnewnew1} into the linear system~\eqref{mdiendekm1} and solving it, we obtain 
\begin{equation}\label{snjdb1}
    \tau_1(\B v_1)=\frac{7 \beta^2 + 18 \beta + 12}{3 \beta^2 + 9 \beta + 6}, \qquad   \tau_1(\B v_2)=\tau_1(\B v_3)=\frac{\beta^2 + 6 \beta + 6 }{3 \beta^2 + 9 \beta + 6}.
\end{equation}
Combining~\eqref{snjdb1} with~\eqref{djeise} and following the notations defined in~\eqref{not1}, we find
\begin{equation*}
\tau_1=r_{\beta}\varphi_1+m_{\beta}\left(\varphi_2+\varphi_3\right) + \phi_1.
\end{equation*}

It remains to prove the expression for $\rho_1$. Since $\rho_1\in \mathbb{P}_2(T)$, by~\eqref{proveqnew}, it can be expressed as 
\begin{equation*}
    \rho_1=\sum_{k=1}^3 \mathcal{L}^{\mathrm{enr}}_k(\rho_1) \varphi_k+\sum_{k=1}^3 \mathcal{I}^{CR}_k(\rho_1) \phi_k.
\end{equation*}
Using~\eqref{propro} and~\eqref{linfunAF3},  we obtain
\begin{equation} \label{dies}
    \rho_1=\sum_{k=1}^3 \rho_1(\B v_k) \varphi_k.
\end{equation}
Now, let us compute the vector
\begin{equation*}
    \left[\rho_1(\B v_1),\rho_1(\B v_2),\rho_1(\B v_3)\right]^T. 
\end{equation*}
For this purpose, leveraging~\eqref{dnnexprimva} of Lemma~\ref{lem12} and~\eqref{propro}, we get
 \begin{equation*}
  \nu_{\beta}  \begin{bmatrix}
2 &  -\sigma_{\beta} & -\sigma_{\beta}\\
-\sigma_{\beta} & 2 & -\sigma_{\beta} \\
-\sigma_{\beta} &-\sigma_{\beta}&  2
\end{bmatrix}
 \begin{bmatrix}
\rho_1(\B v_1)\\
\rho_1(\B v_2)\\
\rho_1(\B v_3)
\end{bmatrix}
= \begin{bmatrix}
1\\
0\\
0
\end{bmatrix}.
\end{equation*}
Thus, we find
 \begin{equation*}
  \frac{1}{\nu_{\beta}L_{\beta}}  \begin{bmatrix}
 \sigma_{\beta}-2 &  -\sigma_{\beta} & -\sigma_{\beta}\\
-\sigma_{\beta} & \sigma_{\beta}-2 & -\sigma_{\beta} \\
-\sigma_{\beta} & -\sigma_{\beta} &   \sigma_{\beta}-2
\end{bmatrix}
 \begin{bmatrix}
1\\
0\\
0
\end{bmatrix}
=  \begin{bmatrix}
\rho_1(\B v_1)\\
\rho_1(\B v_2)\\
\rho_1(\B v_3)
\end{bmatrix},
\end{equation*}
and then
\begin{equation*}
    \rho_1(\B v_1)=\frac{\sigma_{\beta}-2}{\nu_{\beta} L_{\beta}}, \quad  \rho_1(\B v_2)=-\frac{\sigma_{\beta}}{\nu_{\beta}L_{\beta}}, \quad \rho_1(\B v_3)=-\frac{\sigma_{\beta}}{\nu_{\beta}L_{\beta}}.
\end{equation*}
Substituting these values into~\eqref{dies}, we have 
\begin{equation*}
    \rho_1=\frac{1}{\nu_{\beta} L_{\beta}}\left[\left(\sigma_{\beta}-2\right)\varphi_1-\sigma_{\beta} \left(\varphi_2+\varphi_3\right)\right].
\end{equation*}
Analogously, the theorem can be established for $i=2$ and $i=3$; consequently, the thesis follows.
\end{proof}

\begin{theorem}
 Let $\beta>-1$ be a fixed real number. The approximation operator relative to the enriched finite element $\mathcal{E}_{\beta}$
\begin{equation}
\begin{array}{rcl}
{\Pi}_{2,\mathcal{E}_{\beta}}^{{\mathrm{enr}}}: C(T) &\rightarrow& \mathbb{P}_2(T)
\\
f &\mapsto& \displaystyle{\sum_{j=1}^{3}  \mathcal{I}^{\mathrm{CR}}_j(f)\tau_j+ \mathcal{G}^{\mathrm{enr}}_{j,\beta}(f)}\rho_j,
\end{array}
\label{pilinch9Cv}
\end{equation}
reproduces all polynomials of $\mathbb{P}_2(T)$ and satisfies 
\begin{eqnarray*}
\mathcal{I}^{\mathrm{CR}}_j\left({\Pi}_{2,\mathcal{E}_{\beta}}^{\mathrm{enr}}[f]\right)&=&\mathcal{I}^{\mathrm{CR}}_j(f), \qquad j=1,2,3, \\ \mathcal{G}^{\mathrm{enr}}_{j,\beta}\left({\Pi}_{2,\mathcal{E}_{\beta}}^{\mathrm{enr}}[f]\right)&=&\mathcal{G}^{\mathrm{enr}}_{j,\beta}(f), \qquad j=1,2,3.
\end{eqnarray*} 
\end{theorem}
\begin{proof}
The proof is a consequence of~\eqref{proptau} and~\eqref{propro}.
\end{proof}

\section{Numerical experiments}\label{sec4}
In this Section,  we test the effectiveness of the proposed enrichment strategies through multiple examples. We consider the following test functions used in~\cite{Renka1999AAT}
\begin{eqnarray*}
    &&f_1(x,y)=\frac{\sin(2\pi x)\cos(2\pi y)}{2}, \qquad f_2(x,y)=\frac{1}{x^2+y^2+8},\\
&&  f_3(x,y)= \frac{e^{-81/16((x-0.5)^2+(y-0.5)^2)}}{3}, \\ && f_4(x,y)=\frac{\sqrt{64-81((x-0.5)^2+(y-0.5)^2)}}{9}-0.5,\\
&& f_5(x,y)=e^{x+y}, \qquad f_6(x,y)=\frac{1}{x^2+y^2+25},
\end{eqnarray*}
and six different Delaunay triangulations (see Figure~\ref{Fig:regulatri} and Figure~\ref{Fig: DelaunayTrian}). These latter are obtained through the Shewchuk's triangle program~\cite{Shewchuk:1996:TEA}. 

Numerical results are reported in Tables~\ref{TabreTri1}-\ref{TabreTri6}. In these tables, we compare the error in $L^1$-norm produced by the standard Crouzeix--Raviart finite element ($E^{\mathrm{CR}}_1$) with that produced by the enriched approximation operators ${\Pi}_{2,\mathcal{C}_\alpha}^{{\mathrm{enr}}}$, defined in~\eqref{pilinch9C} with $\alpha=1$ ($E^{\mathrm{enr}}_{\mathcal{C}_{\alpha},1}$), and ${\Pi}_{2,\mathcal{E}_{\beta}}^{{\mathrm{enr}}}$, defined in~\eqref{pilinch9Cv} with $\beta=1$ ($E^{\mathrm{enr}}_{\mathcal{E}_{\beta},1}$).

As evident, the approximation generated by the new enriched finite elements outperforms the approximation provided by the traditional Crouzeix--Raviart finite element. This enhancement becomes particularly noteworthy as the number of triangles in a triangulation increases.

 \begin{figure}[ht]
  \centering
    \includegraphics[width=0.3\textwidth]{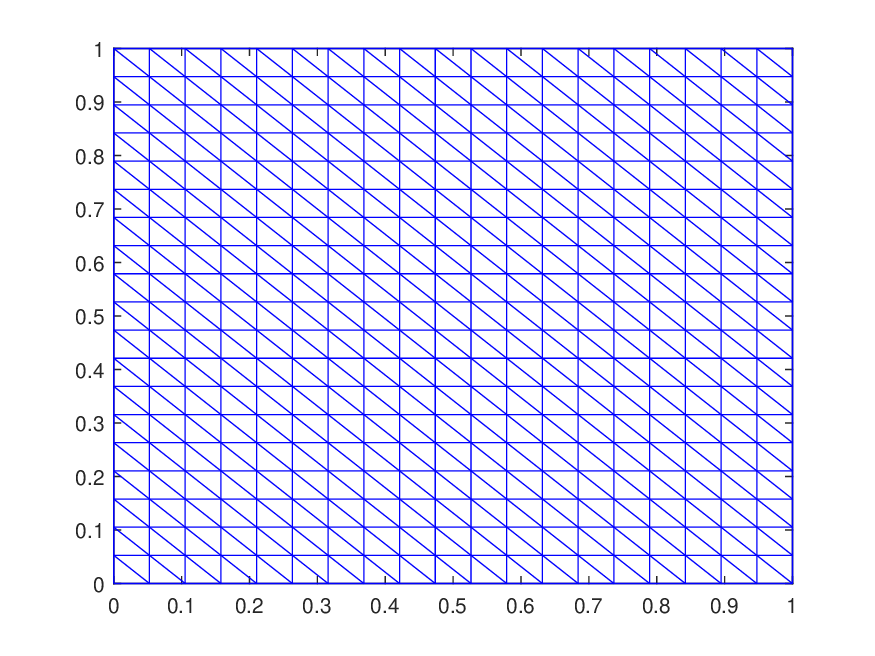} 
        \includegraphics[width=0.3\textwidth]{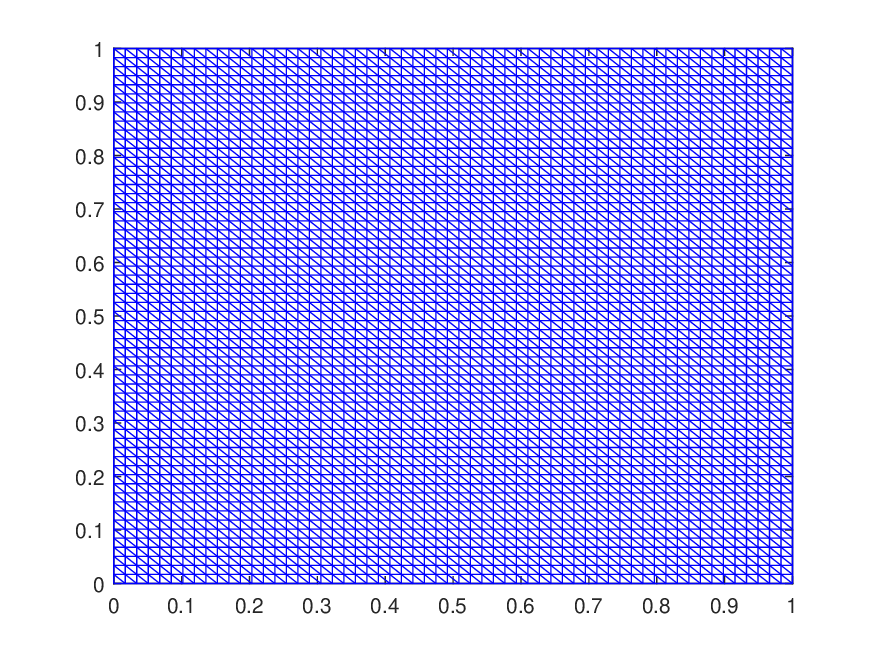} 
             \includegraphics[width=0.3\textwidth]{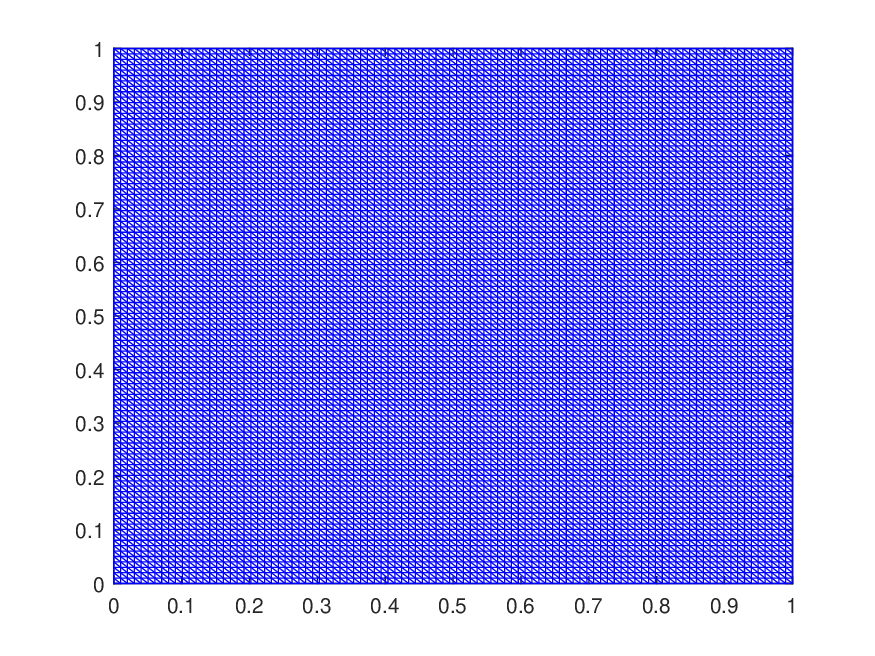} 
         \caption{Delaunay triangulation of $N=722$ (left), $N=6962$ (center) and $N=19602$ (right) tringles.}
          \label{Fig:regulatri}
\end{figure}

 \begin{figure}[ht]
  \centering
    \includegraphics[width=0.3\textwidth]{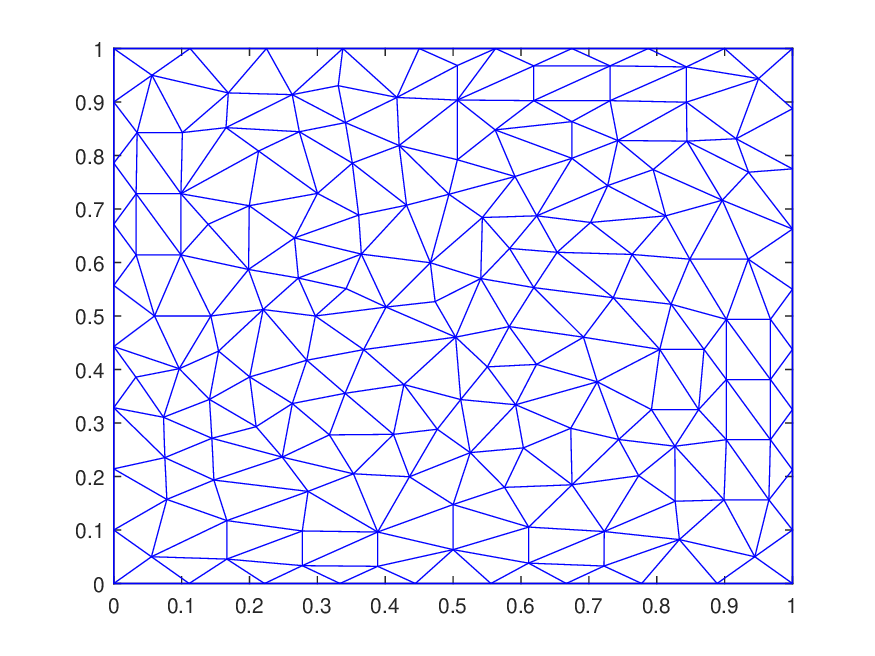} 
        \includegraphics[width=0.3\textwidth]{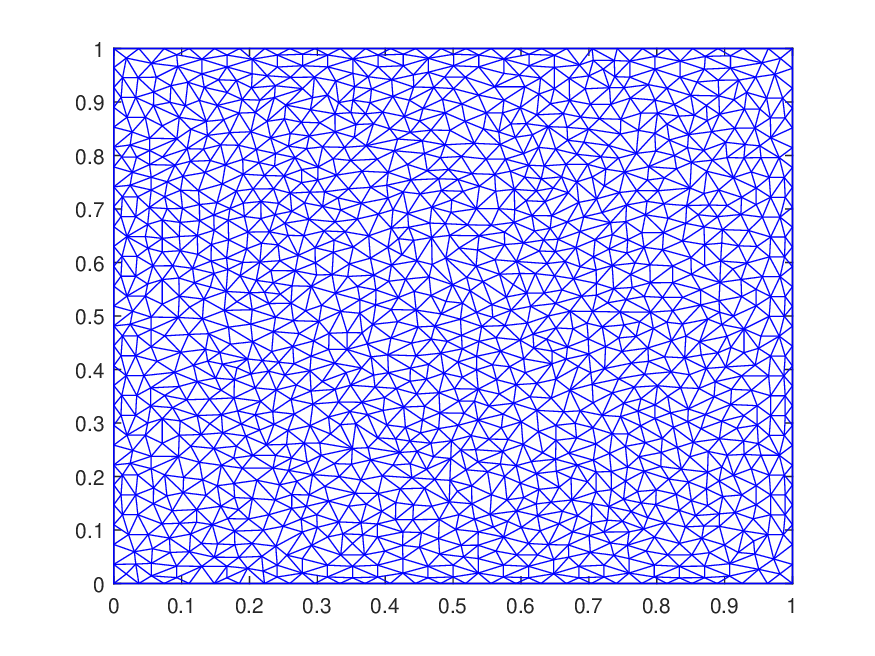}
               \includegraphics[width=0.3\textwidth]{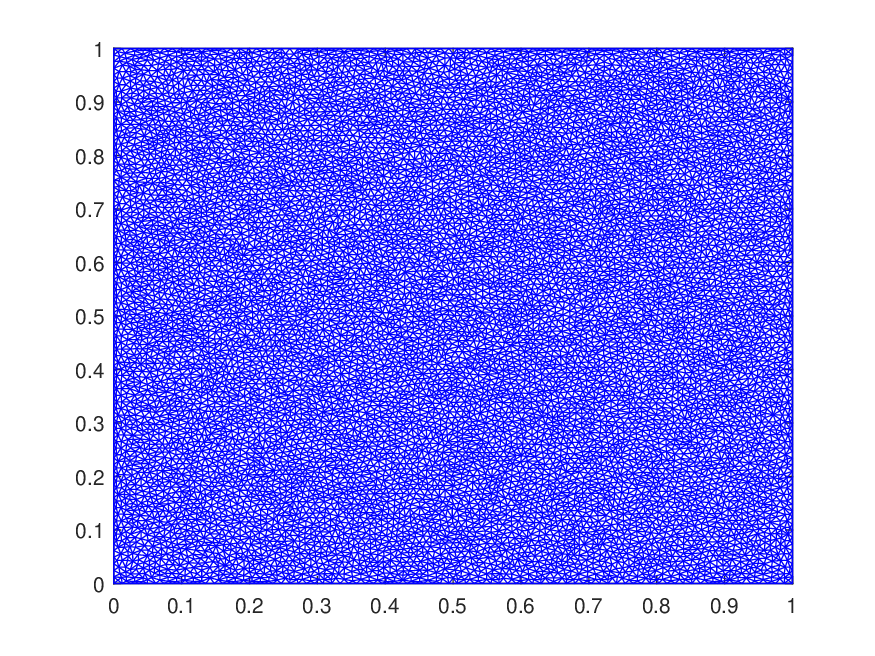}
         \caption{Delaunay triangulation of $N=306$ (left), $N=2650$ (center) and $N=23576$ (right) tringles with no angle smaller than ${20}^{\circ}$.}
          \label{Fig: DelaunayTrian}
\end{figure}

\begin{table}
\centering
   \begin{tabular}{c| c | c | c } 
  &  $E^{\mathrm{CR}}_1$ & $E^{\mathrm{enr}}_{\mathcal{C}\alpha,1}$ & $E^{\mathrm{enr}}_{\mathcal{E}_\beta,1}$ \\ \hline \hline  
$f_1$ & 2.1542e-03 & 9.2333e-05   &    9.3810e-05   \\  
$f_2$ &  4.7276e-06  &  1.0304e-08 & 1.0908e-08       \\  
$f_3$  &  2.2894e-04  &   6.1978e-06 &   6.4967e-06      \\
$f_4$ &  2.8344e-04    &   2.2345e-06   &   2.2874e-06   \\
$f_5$  &  3.1134e-04      &  1.2588e-06&    1.3495e-06   \\
$f_6$ &  5.6037e-07   &    4.2346e-10 &   4.4619e-10   \\ \hline
\end{tabular}
\caption{Comparison between the error in $L^1$-norm produced by the standard Crouzeix--Raviart finite element ($E^{\mathrm{CR}}_1$) with that produced by the enriched approximation operators ${\Pi}_{2,\mathcal{C}_{\alpha}}^{{\mathrm{enr}}}$, defined in~\eqref{pilinch9C} with $\alpha=1$ ($E^{\mathrm{enr}}_{\mathcal{C}_{\alpha},1}$), and ${\Pi}_{2,\mathcal{E}_{\beta}}^{{\mathrm{enr}}}$, defined in~\eqref{pilinch9Cv} with $\beta=1$ ($E^{\mathrm{enr}}_{\mathcal{E}_{\beta},1}$). These comparisons are conducted using the Delaunay triangulation of $N=722$ triangles, see Figure~\ref{Fig:regulatri} (left).}
\label{TabreTri1}
\end{table}

\begin{table}
\centering
   \begin{tabular}{c| c | c | c } 
  &  $E^{\mathrm{CR}}_1$ & $E^{\mathrm{enr}}_{\mathcal{C}\alpha,1}$ & $E^{\mathrm{enr}}_{\mathcal{E}_\beta,1}$ \\ \hline \hline  
$f_1$ & 2.2437e-04 & 3.0933e-06   &   3.1421e-06    \\  
$f_2$ &   4.9026e-07  &  3.4433e-10 &   3.6448e-10       \\  
$f_3$  &  2.3774e-05  &   2.0717e-07 &    2.1716e-07      \\
$f_4$ &  2.9407e-05    &   7.4913e-08   &  7.6675e-08
     \\
$f_5$  &  3.2287e-05      &  4.2038e-08 &  4.5068e-08     \\
$f_6$ &  5.8113e-08   &    1.4150e-11 &  1.4910e-11    \\ \hline
\end{tabular}
\caption{Comparison between the error in $L^1$-norm produced by the standard Crouzeix--Raviart finite element ($E^{\mathrm{CR}}_1$) with that produced by the enriched approximation operators ${\Pi}_{2,\mathcal{C}_{\alpha}}^{{\mathrm{enr}}}$, defined in~\eqref{pilinch9C} with $\alpha=1$ ($E^{\mathrm{enr}}_{\mathcal{C}_{\alpha},1}$), and ${\Pi}_{2,\mathcal{E}_{\beta}}^{{\mathrm{enr}}}$, defined in~\eqref{pilinch9Cv} with $\beta=1$ ($E^{\mathrm{enr}}_{\mathcal{E}_{\beta},1}$).  These comparisons are conducted using the Delaunay triangulation of $N=6962$ triangles, see Figure~\ref{Fig:regulatri} (center).}
\label{TabreTri2}
\end{table}

\begin{table}
\centering
   \begin{tabular}{c| c | c | c } 
  &  $E^{\mathrm{CR}}_1$ & $E^{\mathrm{enr}}_{\mathcal{C}\alpha,1}$ & $E^{\mathrm{enr}}_{\mathcal{E}_\beta,1}$ \\ \hline \hline  
$f_1$ & 7.9716e-05 & 6.5490e-07   &   6.6522e-07    \\  
$f_2$ &  1.7412e-07  &  7.2886e-11 & 7.7151e-11         \\  
$f_3$  &  8.4448e-06  &    4.3854e-08 &   4.5968e-08       \\
$f_4$ &  1.0445e-05    &   1.5861e-08   &   1.6234e-08   \\
$f_5$  &  1.1467e-05     &   8.8980e-09 &  9.5394e-09     \\
$f_6$ &  2.0640e-08   &     2.9953e-12 &   3.1560e-12   \\ \hline
\end{tabular}
\caption{Comparison between the error in $L^1$-norm produced by the standard Crouzeix--Raviart finite element ($E^{\mathrm{CR}}_1$) with that produced by the enriched approximation operators ${\Pi}_{2,\mathcal{C}_{\alpha}}^{{\mathrm{enr}}}$, defined in~\eqref{pilinch9C} with $\alpha=1$ ($E^{\mathrm{enr}}_{\mathcal{C}_{\alpha},1}$), and ${\Pi}_{2,\mathcal{E}_{\beta}}^{{\mathrm{enr}}}$, defined in~\eqref{pilinch9Cv} with $\beta=1$ ($E^{\mathrm{enr}}_{\mathcal{E}_{\beta},1}$). These comparisons are conducted using the Delaunay triangulation of $N=19602$ triangles, see Figure~\ref{Fig:regulatri} (right).}
\label{TabreTri3}
\end{table}

\begin{table}
\centering
   \begin{tabular}{c| c | c | c } 
  &  $E^{\mathrm{CR}}_1$ & $E^{\mathrm{enr}}_{\mathcal{C}\alpha,1}$ & $E^{\mathrm{enr}}_{\mathcal{E}_\beta,1}$ \\ \hline \hline  
$f_1$ & 5.0220e-03 & 3.1623e-04   &    3.3101e-04   \\  
$f_2$ &  1.0948e-05  &   5.9840e-08 & 6.1824e-08      \\  
$f_3$  &  5.5452e-04  &   2.4571e-05 &   2.5882e-05      \\
$f_4$ &  6.9522e-04    &   8.2313e-06   &   8.4936e-06   \\
$f_5$  &  1.5828e-03     &  1.5118e-05 &    1.5723e-05  \\
$f_6$ &  1.3625e-06   &    2.5348e-09 &  2.6162e-09   \\ \hline
\end{tabular}
\caption{Comparison between the error in $L^1$-norm produced by the standard Crouzeix--Raviart finite element ($E^{\mathrm{CR}}_1$) with that produced by the enriched approximation operators ${\Pi}_{2,\mathcal{C}_{\alpha}}^{{\mathrm{enr}}}$, defined in~\eqref{pilinch9C} with $\alpha=1$ ($E^{\mathrm{enr}}_{\mathcal{C}_{\alpha},1}$), and ${\Pi}_{2,\mathcal{E}_{\beta}}^{{\mathrm{enr}}}$, defined in~\eqref{pilinch9Cv} with $\beta=1$ ($E^{\mathrm{enr}}_{\mathcal{E}_{\beta},1}$).  These comparisons are conducted using the Delaunay triangulation of $N=306$ triangles, see Figure~\ref{Fig: DelaunayTrian} (left).}
\label{TabreTri4}
\end{table}

\begin{table}
\centering
   \begin{tabular}{c| c | c | c } 
  &  $E^{\mathrm{CR}}_1$ & $E^{\mathrm{enr}}_{\mathcal{C}\alpha,1}$ & $E^{\mathrm{enr}}_{\mathcal{E}_\beta,1}$ \\ \hline \hline  
$f_1$ &  5.6030e-04 & 1.1742e-05   &    1.2303e-05   \\  
$f_2$ &  1.2010e-06  &   2.1528e-09 & 2.2214e-09      \\  
$f_3$  &  6.1870e-05  &   9.0172e-07 &   9.4571e-07      \\
$f_4$ &  7.6921e-05    &   3.1078e-07   &   3.2010e-07   \\
$f_5$  &  1.7982e-04      &  5.7568e-07 &    5.9968e-07  \\
$f_6$ &  1.4960e-07   &     9.2190e-11 &  9.5052e-11   \\ \hline
\end{tabular}
\caption{Comparison between the error in $L^1$-norm produced by the standard Crouzeix--Raviart finite element ($E^{\mathrm{CR}}_1$) with that produced by the enriched approximation operators ${\Pi}_{2,\mathcal{C}_{\alpha}}^{{\mathrm{enr}}}$, defined in~\eqref{pilinch9C} with $\alpha=1$ ($E^{\mathrm{enr}}_{\mathcal{C}_{\alpha},1}$), and ${\Pi}_{2,\mathcal{E}_{\beta}}^{{\mathrm{enr}}}$, defined in~\eqref{pilinch9Cv} with $\beta=1$ ($E^{\mathrm{enr}}_{\mathcal{E}_{\beta},1}$). These comparisons are conducted using the Delaunay triangulation of $N=2650$ triangles, see Figure~\ref{Fig: DelaunayTrian} (center).}
\label{TabreTri5}
\end{table}

\begin{table}
\centering
   \begin{tabular}{c| c | c | c } 
  &  $E^{\mathrm{CR}}_1$ & $E^{\mathrm{enr}}_{\mathcal{C}\alpha,1}$ & $E^{\mathrm{enr}}_{\mathcal{E}_\beta,1}$ \\ \hline \hline  
$f_1$ &  6.1923e-05 & 4.4045e-07   &    4.6135e-07   \\  
$f_2$ &  1.3353e-07  &   7.8907e-11 & 8.1414e-11      \\  
$f_3$  &  6.8996e-06  &   3.3831e-08 &   3.5476e-08      \\
$f_4$ &  8.6203e-06    &   1.2163e-08   &   1.2538e-08   \\
$f_5$  &  1.9638e-05      &  2.0804e-08 &    2.1686e-08  \\
$f_6$ &  1.6606e-08   &     3.3801e-12 &  3.4848e-12   \\ \hline
\end{tabular}
\caption{Comparison between the error in $L^1$-norm produced by the standard Crouzeix--Raviart finite element ($E^{\mathrm{CR}}_1$) with that produced by the enriched approximation operators ${\Pi}_{2,\mathcal{C}_{\alpha}}^{{\mathrm{enr}}}$, defined in~\eqref{pilinch9C} with $\alpha=1$ ($E^{\mathrm{enr}}_{\mathcal{C}_{\alpha},1}$), and ${\Pi}_{2,\mathcal{E}_{\beta}}^{{\mathrm{enr}}}$, defined in~\eqref{pilinch9Cv} with $\beta=1$ ($E^{\mathrm{enr}}_{\mathcal{E}_{\beta},1}$).  These comparisons are conducted using the Delaunay triangulation of $N=23576$ triangles, see Figure~\ref{Fig: DelaunayTrian} (center).}
\label{TabreTri6}
\end{table}

\section*{Acknowledgments}
This research has been achieved as part of RITA \textquotedblleft Research
 ITalian network on Approximation'' and as part of the UMI group ``Teoria dell'Approssimazione
 e Applicazioni''. The author is a member of the INdAM Research group GNCS.

\bibliographystyle{elsarticle-num}
\bibliography{bibliografia.bib}

\begin{thebibliography}{10}
\expandafter\ifx\csname url\endcsname\relax
  \def\url#1{\texttt{#1}}\fi
\expandafter\ifx\csname urlprefix\endcsname\relax\def\urlprefix{URL }\fi
\expandafter\ifx\csname href\endcsname\relax
  \def\href#1#2{#2} \def\path#1{#1}\fi

\bibitem{ciarlet2002finite}
P.~G. Ciarlet, The finite element method for elliptic problems, SIAM, 2002.

\bibitem{Guessab:2022:SAB}
A.~Guessab, Sharp {A}pproximations based on {D}elaunay {T}riangulations and
  {V}oronoi {D}iagrams, {NSU} {P}ublishing and {P}rinting {C}enter., 2022.

\bibitem{Guessab:2016:AADM}
M.~Bachar, A.~Guessab, {A} {S}imple {N}ecessary and {S}ufficient {C}ondition
  for the {E}nrichment of the {C}rouzeix-{R}aviart {E}lement, Appl. Anal.
  Discret. Math. 10 (2016) 378--393.

\bibitem{Guessab:2016:RM}
M.~Bachar, A.~Guessab, {C}haracterization of the {E}xistence of an {E}nriched
  {L}inear {F}inite {E}lement {A}pproximation {U}sing {B}iorthogonal {S}ystems,
  Results Math. 70 (2016) 401--413.

\bibitem{guessab2017unified}
A.~Guessab, Y.~Zaim, A unified and general framework for enriching finite
  element approximations, Progress in Approximation Theory and Applicable
  Complex Analysis: In Memory of QI Rahman (2017) 491--519.

\bibitem{DellAccio:2022:QFA}
F.~Dell'Accio, F.~Di~Tommaso, A.~Guessab, F.~Nudo, On the improvement of the
  triangular {S}hepard method by non conformal polynomial elements, Appl.
  {N}umer. {M}ath. 184 (2023) 446--460.

\bibitem{DellAccio:2022:AUE}
F.~Dell'Accio, F.~Di~Tommaso, A.~Guessab, F.~Nudo, A unified enrichment
  approach of the standard three-node triangular element, Appl. {N}umer.
  {M}ath. 187 (2023) 1--23.

\bibitem{DellAccio:2022:ESF}
F.~Dell'Accio, F.~Di~Tommaso, A.~Guessab, F.~Nudo, Enrichment strategies for
  the simplicial finite elements, Appl. {M}ath. {C}omput. 451 (2023) 128023.

\bibitem{DellAccio2023AGC}
F.~Dell'Accio, F.~Di~Tommaso, A.~Guessab, F.~Nudo, A general class of enriched
  methods for the simplicial linear finite elements, Appl. {M}ath. {C}omput.
  456 (2023) 128149.

\bibitem{DellAccio2023nuovo}
F.~Dell'Accio, A.~Guessab, F.~Nudo, Improved methods for the enrichment and
  analysis of the simplicial vector-valued linear finite elements, Math.
  Comput. Simul. (2024).
\newblock \href {https://doi.org/https://doi.org/10.1016/j.matcom.2024.01.014}
  {\path{doi:https://doi.org/10.1016/j.matcom.2024.01.014}}.

\bibitem{CR1973:2022:CR}
M.~Crouzeix, P.-A. Raviart, Conforming and nonconforming finite element methods
  for solving the stationary {S}tokes equations {I}, Revue Fran{\c{c}}aise
  d'automatique informatique recherche op{\'e}rationnelle. Math{\'e}matique 7
  (1973) 33--75.

\bibitem{chatzipantelidis1999finite}
P.~Chatzipantelidis, A finite volume method based on the {C}rouzeix--{R}aviart
  element for elliptic {PDE}'s in two dimensions, Numer. Math. 82 (1999)
  409--432.

\bibitem{hansbo2003discontinuous}
P.~Hansbo, M.~G. Larson, Discontinuous {G}alerkin and the {C}rouzeix--{R}aviart
  element: application to elasticity, Esaim Math. Model Numer. Anal. 37 (2003)
  63--72.

\bibitem{burman2005stabilized}
E.~Burman, P.~Hansbo, Stabilized {C}rouzeix--{R}aviart element for the
  {D}arcy-{S}tokes problem, Numer. Methods Partial Differ. Equ. 21 (2005)
  986--997.

\bibitem{zhu2014analysis}
Y.~Zhu, Analysis of a multigrid preconditioner for {C}rouzeix--{R}aviart
  discretization of elliptic partial differential equation with jump
  coefficients, Numer. Linear. Algebra Appl. 21 (2014) 24--38.

\bibitem{younes2014combination}
A.~Younes, A.~Makradi, A.~Zidane, Q.~Shao, L.~Bouhala, A combination of
  {C}rouzeix-{R}aviart, discontinuous {G}alerkin and {MPFA} methods for
  buoyancy-driven flows, Int. J. Numer. Methods Heat Fluid Flow. 24 (2014)
  735--759.

\bibitem{di2015extension}
D.~Di~Pietro, S.~Lemaire, An extension of the {C}rouzeix--{R}aviart space to
  general meshes with application to quasi-incompressible linear elasticity and
  {S}tokes flow, Math. Comput. 84 (2015) 1--31.

\bibitem{ve2019quasi}
R.~Verf\"urth, P.~Zanotti, A {Q}uasi-optimal {C}rouzeix--{R}aviart
  discretization of the {S}tokes equations, SIAM J. Numer. Anal. 57 (2019)
  1082--1099.

\bibitem{chen2019finite}
S.~Chen, X.~Li, H.~Rui, The finite volume method based on the
  {C}rouzeix--{R}aviart element for a fracture model, Numer. Methods Partial
  Differ. Equ. 35 (2019) 1904--1927.

\bibitem{Davis:1975:IAA}
P.~J. Davis, Interpolation and approximation, Dover Publications, 1975.

\bibitem{Abramowitz:1948:HOM}
M.~Abramowitz, I.~A. Stegun, Handbook of {M}athematical {F}unctions with
  {F}ormulas, {G}raphs, and {M}athematical {T}ables, {US} {G}overnment printing
  office, 1948.

\bibitem{Renka1999AAT}
R.~J. Renka, R.~Brown, Algorithm 792: accuracy test of {ACM} algorithms for
  interpolation of scattered data in the plane, {ACM} {T}ransactions on
  {M}athematical {S}oftware ({TOMS}) 25 (1999) 78--94.

\bibitem{Shewchuk:1996:TEA}
J.~R. Shewchuk, Triangle: {E}ngineering a {2D} {Q}uality {M}esh {G}enerator and
  {D}elaunay {T}riangulator, in: M.~C. Lin, D.~Manocha (Eds.), Applied
  Computational Geometry: Towards Geometric Engineering, Vol. 1148 of Lecture
  Notes in Computer Science, Springer-Verlag, Berlin, Heidelberg, 1996, pp.
  203--222.

\end{thebibliography}

\end{document}